\documentclass{amsart}

\title{Knot Colorings: Coloring and Goeritz matrices}
\author{Sudipta Kolay }
\address{School of Mathematics \\ Georgia Institute of Technology\\Atlanta, GA 30332, USA}
\email{skolay3@math.gatech.edu}
\date{}

\usepackage[utf8x]{inputenc}

\usepackage{utopia}

\usepackage[a4paper,top=3cm,bottom=2cm,left=3cm,right=3cm,marginparwidth=1.75cm]{geometry}

\usepackage{amsmath,amsthm,amsfonts,amssymb}
\usepackage{bbm}
\usepackage{graphicx}
\usepackage[colorlinks=true, allcolors=blue]{hyperref}
\usepackage{enumitem} 
\usepackage{spalign}

\theoremstyle{plain}
\newtheorem{thm}{Theorem}

\newtheorem{prop}[thm]{Proposition}

\newtheorem{claim}[thm]{Claim}
\newtheorem{ex}[thm]{Exercise}

\theoremstyle{definition}

\theoremstyle{remark}

\newtheorem{example}[thm]{Example}

\begin{document}

\begin{abstract}
    Knot colorings are one of the simplest ways to distinguish knots, dating back to Reidemeister, and popularized by Fox. In this mostly expository article, we discuss knot invariants
    like colorability, knot determinant and number of colorings, and how these can be computed from either the coloring matrix or the Goeritz matrix. We give an elementary approach
    to this equivalence, without using any algebraic topology. We also compute knot determinant, nullity of pretzel knots with arbitrarily many twist regions. 
\end{abstract}

\maketitle

\section{Introduction}

The purpose of this note is to give an exposition of knot colorings through coloring and Goeritz matrices, and to discuss of $n$-colorings of pretzel knots. Although the coloring and Goeritz matrices can be combinatorially defined given a diagram of the knot, the classical proofs of this equality of their determinants, and related results in the literature \cite[Chapter 9]{L} require some background on the part of reader. We give an elementary proof of these results without using any algebraic topology. Our approach here is somewhat similar to that of \cite{T} (building on \cite{C}), although more elementary. While most of this article is expository, the results (and our methods of computation) in the last section for the determinant and number of colorings of pretzel knots with arbitrarily many twist regions are new. It is hoped that this article would be helpful to people at all levels interested in learning about knots and knot colorings.

A knot is a smooth embedding of a circle $S^1$ in three-space $\mathbb{R}^3$. Here are two examples:

\begin{figure}[!ht]
    \centering
    \includegraphics[width=8 cm]{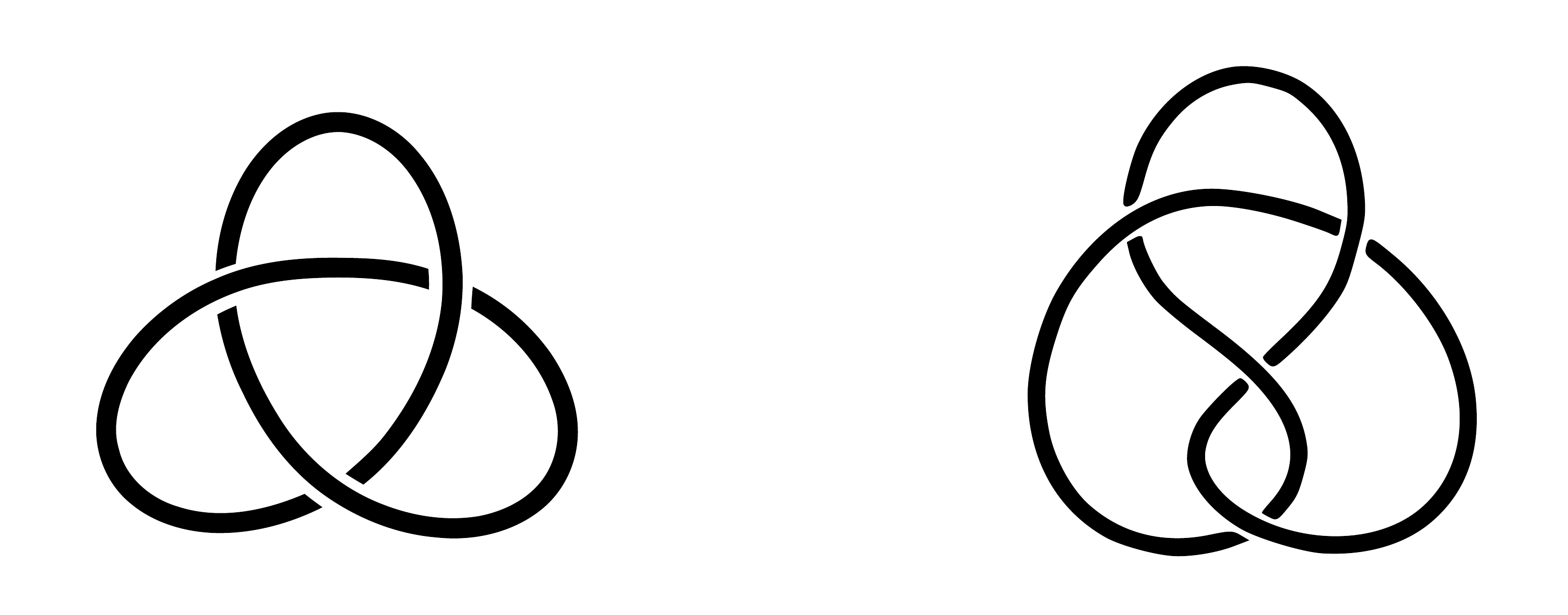}
    \caption{The trefoil and the figure eight knots}
    \label{knotegs}
\end{figure}

 The fundamental question in knot theory is given two knots, can we tell them apart? One of the simplest ways of telling knots apart is \emph{tricolorability}, whether we can non-trivially (i.e. using at least two colors) color the knot with three colors so that at any crossing the three strands coming together has either all the same colors, or three different colors. This is the simplest invariant which tells the above two knots, the trefoil and the figure eight knot, are different (not isotopic, meaning one cannot be deformed into another).
 
 The notion of tricolorability goes back at least to Reidemeister, at the beginning of the nineteenth century. Ralph Fox generalized this notion to $n$-colorability, and popularized this way of studying knots. The property of being $n$-colorable (and related invariants, like total number of $n$-colorings) is a fairly powerful knot invariant, which can be developed with minimal amount of mathematical machinery, namely linear algebra.

There is a related integer valued knot invariant called the determinant of a knot whose divisors $n$ are exactly those for which the given knot is $n$-colorable. It is well known that this invariant determinant is the absolute value of the determinant of any coloring matrix of the knot and it is also the determinant of any Goeritz matrix for the knot. We will give an elementary proof of this result in Sections 6 and 7, after reviewing necessary background materials up to Section 5. We will use ideas from Section 6 to set up an equivalent linear system for pretzel knots in the final section.

\noindent \textit{Acknowledgements}. The author would thanks John Etnyre for making helpful comments on earlier drafts of this paper. This work is partially supported by NSF grants DMS-1608684 and DMS-1906414.

\section{Knots and links}

A link is a smooth embedding (or in other words an injective map which is differentiable at each point) of disjoint union of circles in $\mathbb{R}^3$ (so a knot is just a link with one component). Given any such embedding, we can orthogonally project everything with respect to the $z$-axis, and we get a projection of the knot in the $xy$-plane, which we will call the link diagram. Generically\footnote{It may happen that some knot diagram has triple (or higher order) crossings, however we can isotope some of the strands so that the only crossings that remains are double points.}, we may assume that the only non-embedded points in the diagram are isolated double points. The advantage of knot diagrams is that it is much easier to draw on a piece of paper and manipulate it. In fact, all the examples of knots and links given here (or any other paper based format) is a diagram. While it is clear that given any link in $\mathbb{R}^3$ we have a diagram for it, one may wonder if we can say when two diagrams represent the same knot. The answer is yes, as proven by Alexander-Briggs \cite{AB} and independently Reidemeister \cite{R}.

Any two knot diagrams of the same link, up to planar isotopy, can be related by a sequence of the following three moves (called the Reidemeister moves). Each move operates on a small region of the diagram and is one of three types (see Figure ~\ref{RM}):
\begin{enumerate}
    \item Twist (or untiwst) a strand.
    \item Isotope a strand over (or under) another strand.
    \item Isotope a strand completely over (or under) a crossing.
\end{enumerate}
\begin{figure}[!ht]
    \centering
    \includegraphics[width=14 cm]{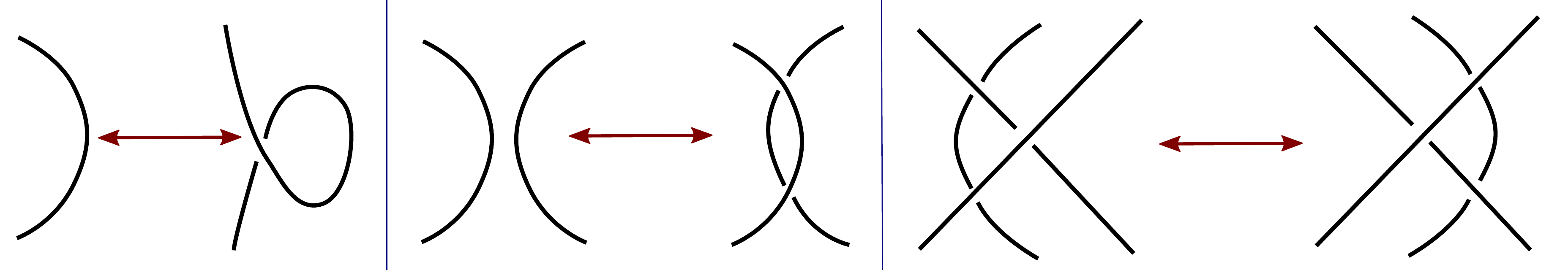}
    \caption{Reidemeister moves}
    \label{RM}
\end{figure}
The above result shows it is enough to understand links diagrams to understand links. 
Moreover, if we can define an object from a diagram, and show it remains unchanged under any of the Reidemeister moves, it follows that it is in fact an invariant of the link (and not the diagram).

\section{Coloring of knots and links}
Now we are ready to introduce $n$-colorings of links. Given any knot diagram $D$, an $n$-\textbf{coloring} of a link is is an assignment to each strand
an element of $\mathbb{Z}/n\mathbb{Z}$, such that whenever we have a crossing with the overstrand associated to $y$, and the understrands associated to $x$ or $z$ as illustrated by Figure \ref{kcolor}, we have $x+z=2y$ in $\mathbb{Z}/n\mathbb{Z}$.
\begin{figure}[!ht]
    \centering
    \includegraphics[width=6 cm]{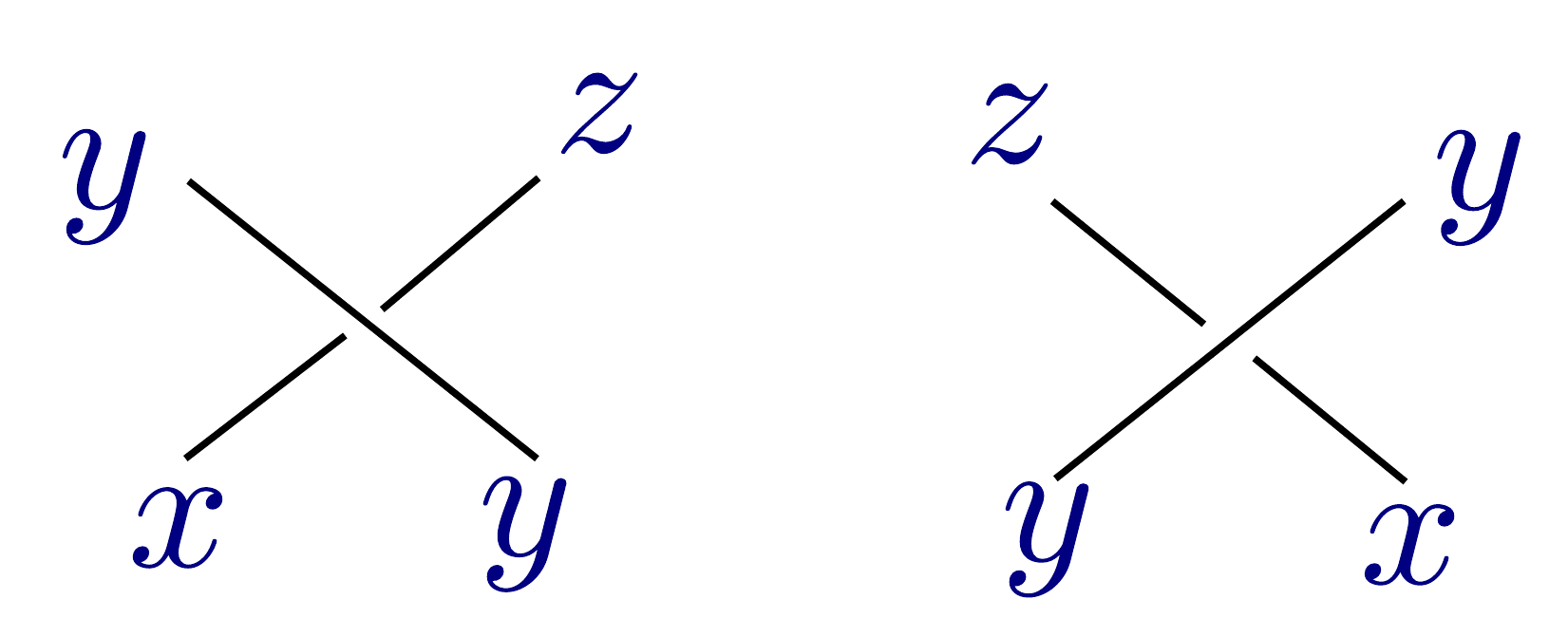}
    \caption{Coloring equation: $2y=x+z$.}
    \label{kcolor}
\end{figure}

\begin{ex}
Show that given any knot diagram with such a coloring, if a Reidemeister move is applied the new diagram has a unique coloring which is the same outside the small region where the move was applied. Conclude that an $n$-coloring is in fact an invariant of the link (and not just the diagram).
\end{ex}
If $n$ is small it is customary to color the strands with $n$ distinct colors, and this is where the name colorability comes from. For example when $n=3$, we use three colors customarily red, blue and green; and $3$-colorability is also known as tricolorability. 
\begin{figure}[!ht]
    \centering
    \includegraphics[width=9 cm]{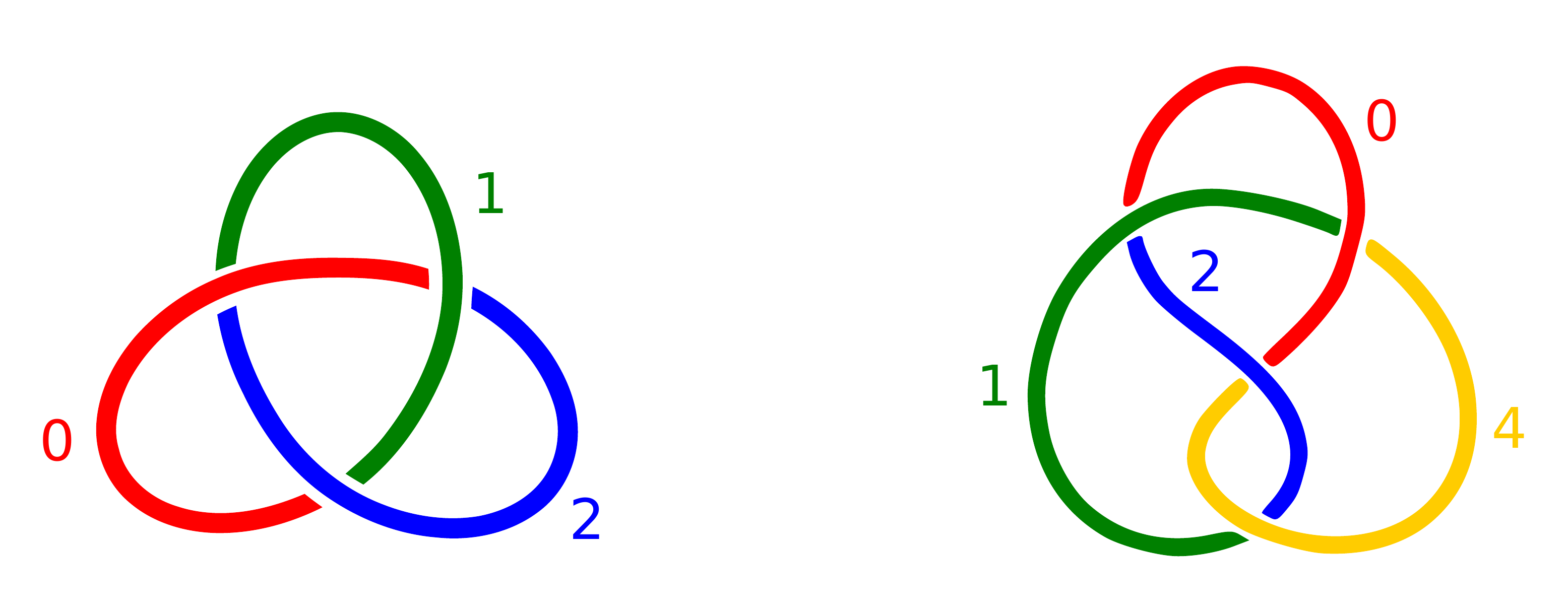}
    \caption{Examples of a 3-coloring and a 5-coloring}
    \label{CE}
\end{figure}
When $n$ is large, it is not practical to use different colors, and we simply use the labelling (perhaps this invariant could also have been called $n$-labelling) on each of the strands.
Note that if we color (or label) each strand by constant element $s\in\mathbb{Z}/n\mathbb{Z}$, then all the above constraints are trivially satisfied. Such a coloring is called a \emph{trivial} coloring, and any link has $n$ such trivial $n$-colorings. Hence we will say a link is $n$-colorable if it has a non-trivial $n$-coloring (because otherwise every link has a coloring and this notion would not be able to distinguish between links). We will see in the next section that the set of all $n$-colorings (including the trivial ones) has additional structure, and this is also an invariant (see Exercise \ref{ex2}). Had we removed the trivial colorings from this collection, the remaining set does not canonically get such a structure.

For the rest of this article, we will restrict to the case of knots, because some of the statements that follow would get very technical otherwise. The technical issues come from the fact that one has more trivial colorings for split links (i.e. links which can be split apart to lie inside disjoint three balls in $\mathbb{R}^3$), and to get appropriate count of number of different colorings we need to divide out by the number of trivial colorings. All of the statements hold for non-split links (i.e. links which not split links). If we had a split link, we could separate them into different regions and work out coloring invariants for each of the separate non-split links. This would suggest that there is no loss of generality to restricting to the case of non-split links, except for the caveat that given a link diagram, it is not easy to determine if it is a diagram of a split link or not; or if it is split, how many components it has.

\section{Coloring matrix}
Notice that the constraints we had at each crossing for $n$-colorings is a linear equation, which suggests that we can consider all such equations together and use tools from linear algebra to gain a better understanding of what is going on.

Given a knot diagram $D$ with $c$ crossings, we have exactly $c$ strands
(with the only exception being the standard diagram of the unknot, which has one strand and no crossings). For each strand let us assign a variable $x_i$, and for each crossing we can write down an equation $x_i+x_k=2 x_j$, which has to be satisfied if the $x_i$'s gave rise to a valid coloring.
If we consider the set of linear equations, and write out the matrix form, we get 
$$\tilde{C}\vec x=\vec 0,$$ where $\tilde{C}$ is a $c\times c$ matrix, and $\vec x$ is a column vector with $i$-th entry $x_i$. We define $\tilde{C}$ to be the \emph{pre-coloring matrix} for the diagram $D$.

\begin{figure}[!ht]
    \centering
    \includegraphics[width=12 cm]{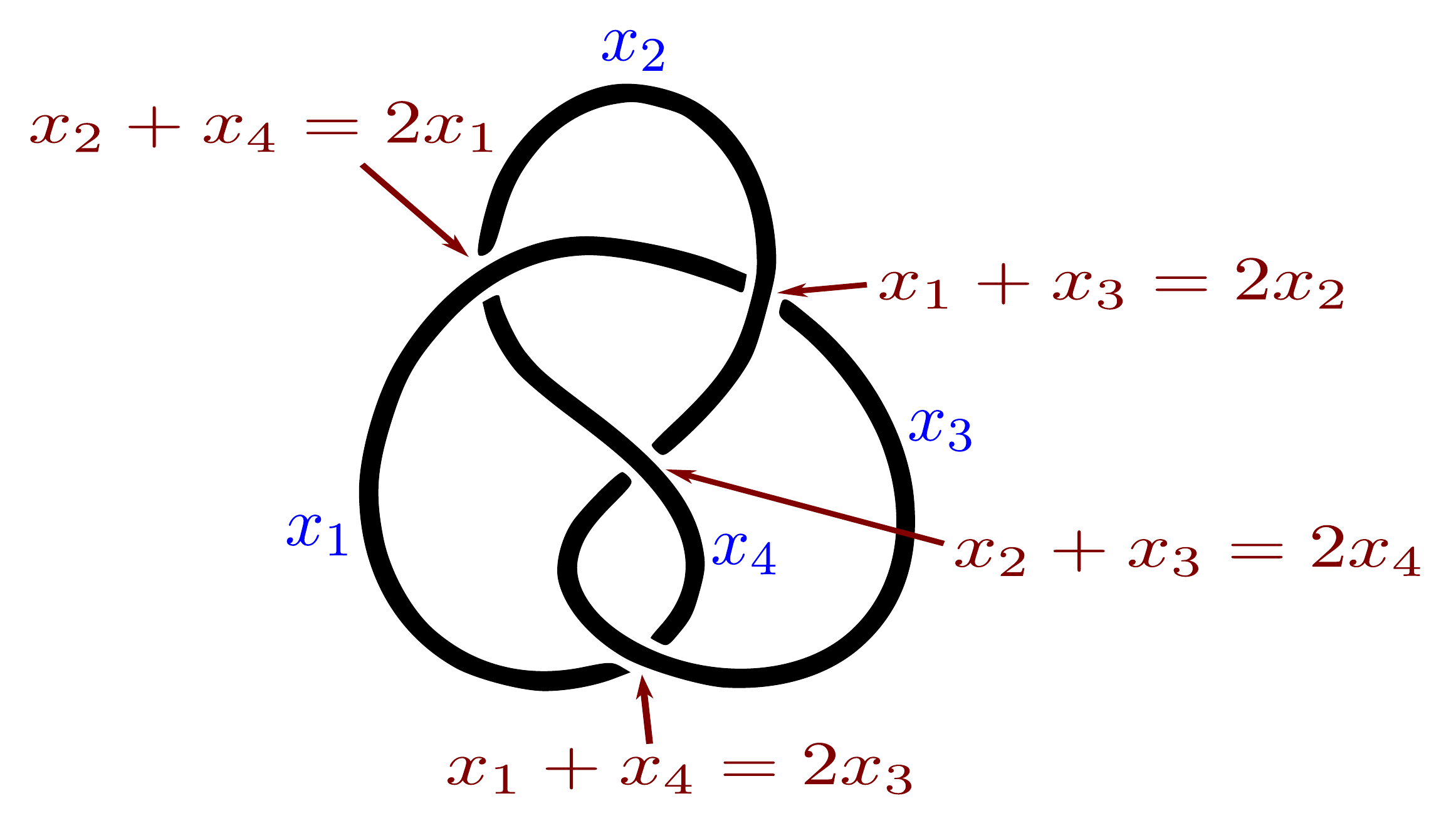}
    \caption{Coloring equations for the figure eight knot.}
    \label{fig8}
\end{figure}

\begin{example}\label{eg1} In this example we work out the pre-coloring matrix of the figure eight knot, see Figure \ref{fig8}. The system of equations we obtain are:
\[ \spalignsys{x_1-2x_2+x_3=0;-2x_1+x_2+x_4=0;x_1-2x_3+x_4=0;x_2+x_3-2x_4=0}\]
Hence the pre-coloring  matrix $\tilde{C}$ equals $$\spalignmat{1 -2 1 0; -2 1 0 1; 1 0 -2 1; 0 1 1 -2}.$$
Notice that we made some choices, we could have labelled the strands in a different order, or permuted the equations in the linear system, and we would have obtained a different pre-coloring matrix.
\end{example}

We observe that the set of all possible $n$-colorings is nothing but the solution space of the above equation reduced modulo $n$, i.e. the Null Space of the pre-coloring matrix, and as such gets the structure of a vector subspace.
\begin{ex}\label{ex1}
Go back to your solution of Exercise \ref{ex1} and check that the linear structure of the set of all $n$-colorings is preserved under each of the Reidemeister moves. 
Conclude that the solution space is an invariant of the knot.
\end{ex}

We note that the sum of the columns of these matrices are $\vec 0$, and this is in fact true more generally (possibly with a signed sum, depending on if we write
the equation as $x_i+x_k-2 x_j=0$ or $-x_i-x_k+2 x_j=0$).  In other words the constant vector $\vec s$ (with all entries $s$) is a solution to the system $\tilde{C}\vec x=\vec 0,$
and we will call such solutions \emph{semitrivial}. Hence if $\vec x$ is any solution, $\vec x+ \vec s$ is again a solution for any scalar $s$. We can fix this redundancy,
by requiring one of the $x_i$'s is zero.  Since the pre-coloring matrix always has a non-trivial solution, it is singular, we know that there is some linear dependence among the rows
as well. So let us delete some row and some column from the pre-coloring matrix, and we will call this resulting square matrix $C$ \emph{a} \textbf{coloring matrix} for the diagram $D$.
While the matrix $C$ is not an invariant (even for the knot diagram $D$), it turns out the absolute value of the determinant of $C$ is an invariant of the knot (not just a diagram!), which we will call \emph{the} \textbf{determinant} of the knot.\\

We digress to some linear algebra to see why various cofactors of pre-coloring matrices are the same.
Recall, for any square matrix $A$, the $(i,j)$ cofactor $C_{i,j}$ is $(-1)^{ij}$ times the determinant of the matrix $A_{i,j}$ obtained by removing the $i$-th row and $j$-th column from A. \\

\begin{prop}\label{prop1}
Suppose $A$ is an $n\times n$ matrix with real entries  (or more generally with entries in some commutative ring) so that sum of all the columns is the column vector of all zeros, and sum of all the rows is the row vector of all zeros (equivalently, the sum of entries of each row and each column add up to zero). Then all the $(i,j)$ cofactors of $A$ are the same.
\end{prop}

\begin{ex}\label{ex2}

Work out the following outline to prove the above proposition:
\begin{enumerate}
    \item For any such matrix $A$, show that you can reconstruct $A$ if you know the submatrix $A_{i,j}$, and the indices $i,j$.
    \item Using $\vec R_1+...+\vec R_n=\vec 0$, and properties of the determinant, relate $C_{1,1}$ and $C_{1,2}$. 
    \item Relate $C_{1,1}$ with $C_{1,j}$.
    \item Relate $C_{1,1}$ with $C_{i,j}$.
\end{enumerate} 
\end{ex}

In order to use this result from linear algebra, we need to know that
for an arbitrary knot diagram we can choose a pre-coloring matrix in the specified form, just as we saw in the example.
\begin{ex} \label{ex3}
Given any knot diagram $D$, show that one can number the crossings and strands in such a way so that hypothesis of Proposition \ref{prop1} holds.
\end{ex}

Another approach is to find a modified version of Proposition \ref{prop1} which can be applied to any pre-coloring matrix.
\begin{ex} \label{ex4}
Suppose $A$ is an $n\times n$ matrix with real entries so that 
a linear combination (with each weight being $\pm 1$) of every column is the column vector of all zeros, a linear combination (with each weight being $\pm 1$) of every row is the row vector of all zeros. Then all the absolute values of cofactors (or equivalently, minors) of $A$ are the same.
\end{ex}

The aforementioned results from linear algebra tell us that the determinant is indeed an invariant of the diagram, but we still need to justify that it does not depend on the diagram. One way (which is commonly used in the literature) to go about this is to show that the knot determinant is invariant under the Reidemeister moves. We will take a slightly different approach here, by invoking the fact (as seen in the last section) that $n$-colorability 
is a knot invariant.

\begin{prop}\label{prop2}
The following are equivalent:
\begin{enumerate}[label=(\roman*)]
    \item A pre-coloring matrix $\tilde{C}$ has a non-semitrivial solution in $\mathbb{Z}/n\mathbb{Z}$.
    \item A coloring matrix ${C}$ has a non-trivial solution in $\mathbb{Z}/n\mathbb{Z}$.
    \item $\det C=0$ in $\mathbb{Z}/n\mathbb{Z}$.
    \item $n$ divides $\det C$.
\end{enumerate}
\end{prop}

\begin{proof}
 Suppose $C$ is obtained from $\tilde{C}$ by deleting the $i$-th row and $j$-th column.\\
$(i)\Longleftrightarrow(ii)$: Given a solution $\vec x$ of the pre-coloring matrix $\tilde{C}$, we can add a constant vector $\vec s$ to get another solution $\vec y$ of $\tilde{C}$ with $y_j=0$ (note that this operation preserves if the solution vector was semitrivial or not). The column vector obtained by deleting $y_j$ from $\vec y$ is a solution to the coloring matrix $C$, and moreover one can go backwards by adding $0$ to the $j$-th spot. This process describes a bijection between all solutions of $\tilde{C}$ with $j$-th entry zero, and all solutions of $C$, and further the non-semitrivial solutions of $\tilde{C}$ correspond exactly to the non-trivial solutions of $C$.\\
$(ii)\Longleftrightarrow(iii)$ is a standard fact in linear algebra; and $(iii)\Longleftrightarrow(iv)$ follows from the definition of the quotient set $\mathbb{Z}/n\mathbb{Z}$.
\end{proof}

Suppose $D_1$ and $D_2$ are two diagrams for a knot $K$, with coloring matrices $C_1$ and $C_2$. Then for every $n$ dividing $\det(D_1)$, we see by Proposition 6 that the diagram $D_1$ is $n$-colorable, and hence (since $n$-colorability 
is a knot invariant) the diagram $D_2$ is $n$-colorable as well, and again by Proposition 6 we have $n$ divides $\det(D_2)$. So we see that $\det(D_1)$ divides $\det(D_2)$, and by  symmetry $\det(D_2)$ divides $\det(D_1)$. Thus the absolute values of $\det(D_1)$ and $\det(D_2)$ are the same, so knot determinant is  invariant of knot.

The determinant of knot $K$ determines for which $n$, the knot $K$ is $n$-colorable. However the reader may have noticed that in order to find the knot determinant directly from definition, we need to compute the determinant of a matrix of size $(c-1)\times (c-1)$, which in some cases may not be computationally efficient. In the next sections, we will talk about an alternate, easier approach to compute the knot determinant and related invariants.

\section{Goeritz matrix} Given a knot diagram we can checkerboard color the regions, for example see
Figure ~\ref{GM}. Suppose we assign to each crossing $c$ the sign $\eta(c)\in\{\pm 1 \}$ according to Figure ~\ref{SC}: 
\begin{figure}[!ht]
    \centering
    \includegraphics[width=7 cm]{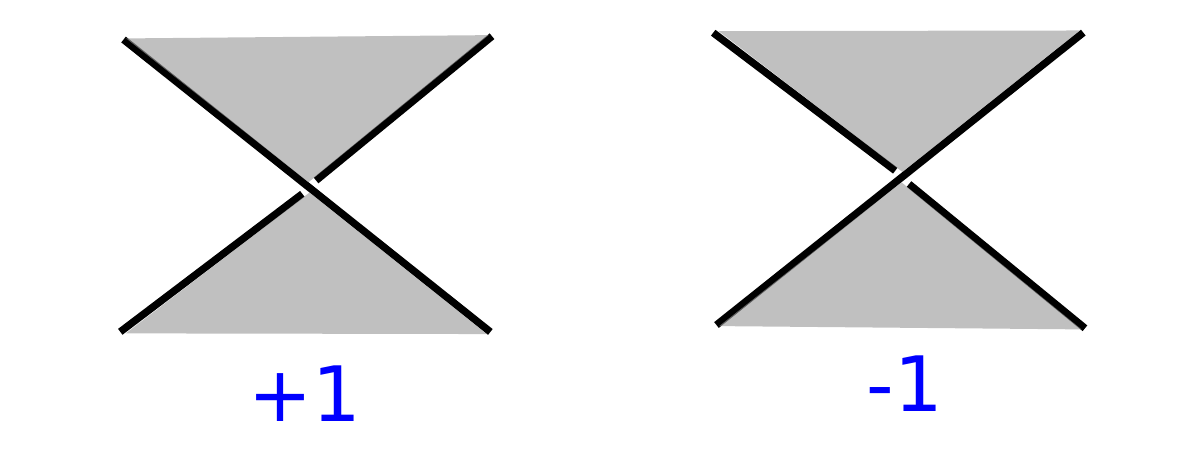}
    \caption{Signs at crossings. Note that these signs are independent of an orientation on the knot, but will flip if we switch the shaded and unshaded regions in the checkerboard coloring.}
    \label{SC}
\end{figure}

Suppose the shaded regions are enumerated $R_1,...,R_k$.
Let us define the $k\times k$ pre-Goeritz matrix $\tilde{G}$ for the diagram $D$ by the following:
For the off diagonal entries ($i\neq j$), we set  $$\tilde{G}_{i,j}:= \sum \eta (c),$$ where the sum ranges over all crossings $c$ where regions $R_i$ and $R_j$ come together. Let us now define the diagonal entries by 
$$\tilde{G}_{i,i}:=-\sum_{j\neq i}\tilde{G}_{i,j}.$$ We now note that $\tilde{G}$ is a square symmetric matrix whose columns (respectively rows) sum to zero.

\begin{figure}[!ht]
    \centering
    \includegraphics[width=10 cm]{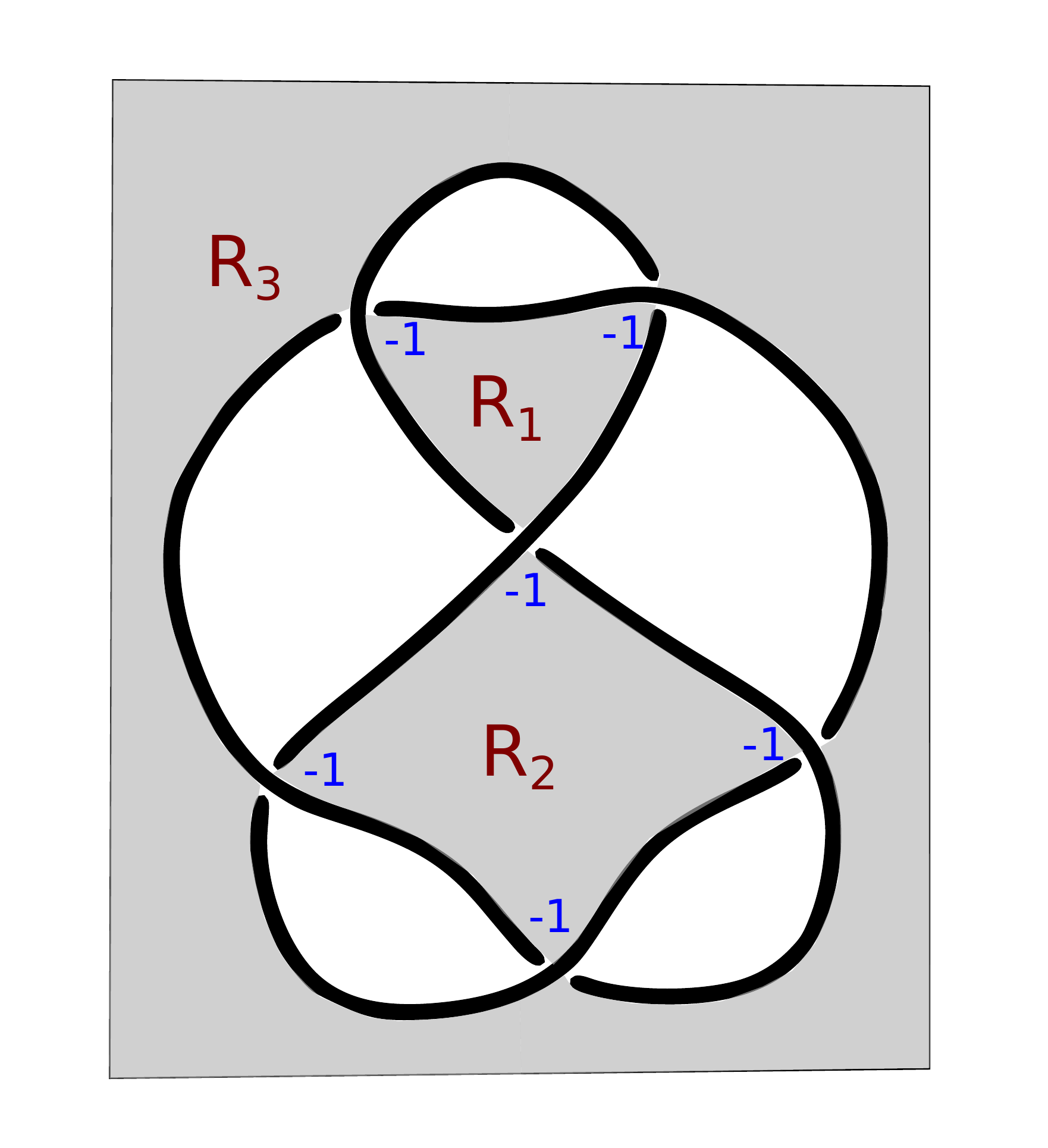}
    \caption{A knot diagram with a checkerboard coloring.}
    \label{GM}
\end{figure}

\begin{example}\label{eg2}
The pre-Goeritz matrix for the knot diagram in Figure \ref{GM} with the given checkerboard coloring is:
$$\spalignmat{3 -1 -2;-1 4 -3;-2 -3 5}$$

\end{example}

Just as we saw earlier for the pre-coloring matrices, the determinant of $\tilde{G}$ is zero, however if we form a Goeritz matrix by deleting any row and any column, the absolute value of determinant the resulting matrix is well defined for the diagram, which we will call \emph{the} \textbf{Goeritz determinant} of the diagram. We get a similar result regarding the Goeritz matrices as we had for the coloring matrices, and the same proof works.

\begin{prop}\label{prop3} Let us pick any  pre-Goeritz matrix $\tilde{G}$ for a knot diagram $D$, and obtain a Goeritz matrix $G$ by deleting some row and some column.
The following are equivalent:
\begin{enumerate}[label=(\roman*)]
    \item $\tilde{G}$ has a non-semitrivial solution in $\mathbb{Z}/n\mathbb{Z}$.
    \item ${G}$ has a non-trivial solution in $\mathbb{Z}/n\mathbb{Z}$.
    \item $\det G=0$ in $\mathbb{Z}/n\mathbb{Z}$.
    \item $n$ divides $\det G$.
\end{enumerate}
\end{prop}

As the reader can probably expect, the Goeritz determinant does not depend on the knot diagram. In fact,  we will show it is exactly the same as the determinant
of the knot (coming from coloring matrix)! This  result also means that determinant of a Goeritz matrix obtained from the shaded regions is same in absolute value to the
one coming from the unshaded regions. While it is not too difficult to see that the Goeritz determinant is invariant under the Reidemeister moves,
it is not quite clear from this perspective if the determinants corresponding to the unshaded and shaded regions are related, or if they are related to the determinant
of a coloring matrix. In the next two sections, we will describe another approach where we will create a bijection between the solution spaces of the pre-coloring and pre-Goeritz matrices, and this will give us the above result about determinants, and a bit more.

\section{Difference} 
Let us note that if we have a colored knot diagram with a collection of half twists as illustrated in Figure \ref{df}, 
\begin{figure}[!ht]
    \centering
    \includegraphics[width=12 cm]{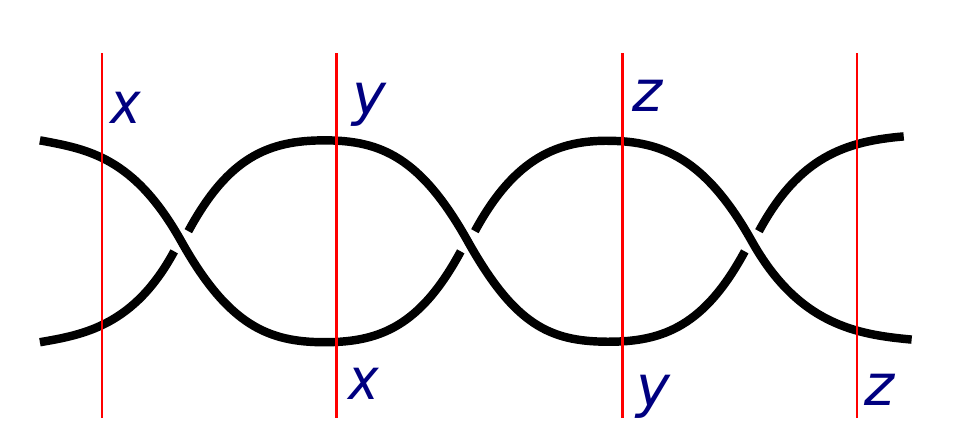}
    \caption{Difference}
    \label{df}
\end{figure}
the difference between the colors on top and the bottom is the same for every generic vertical slice (i.e. does not pass through a crossing or {touch}\footnote{By two curves touching at a point we will mean there is a point of intersection where the tangent lines agree.} the knot diagram), since the coloring equation is equivalent to $x_i-x_j= x_j-x_k$. In other words, if we drew any generic arc between the regions $R_1$ and $R_2$ in the part of the diagram illustrated above, and took the difference of the colors, then we would get the same value. One may wonder if something similar works for any two crossing-adjacent shaded (or unshaded) regions. Indeed, this will be the case more generally (that is for non adjacent regions, even with different checkerboard colors), once we formulate the right generalization for this difference.

Given a knot diagram $D$ with regions $R_1,...,R_r$ (ignoring checkerboard coloring for now), of a knot $K$ with an $n$-coloring, let us define the \textbf{difference} $d(R_i,R_j)\in\mathbb{Z}/n\mathbb{Z}$ between any two regions
by taking any generic (i.e. not passing through a crossing, and intersecting $D$ transversely) oriented simple 
arc  between points in the interior of regions $R_i$ and $R_j$, and computing the alternate (beginning with positive) sum of the colors on the strand the 
arc crosses.

\begin{claim}\label{wd} The difference $d(R_i,R_j)$ is well defined, that is it does not depend on the arc we chose to compute the alternating sum.
\end{claim}
We sketch two proofs of this well definedness below, after we discuss an example.

\begin{example}\label{egdiff}
 Let us work out the differences in 5-coloring of the figure eight knot we saw earlier, see Figure \ref{fig8diff}.
 Notice that for the arc $\gamma_1$ between regions $R_1$ and $R_6$, there are three intersection points, and the alternating sum is
 $2-4+1=-1=4$ in $\mathbb{Z}/5\mathbb{Z}$, where as there is only one intersection for the arc $\gamma_2$ and here the alternating sum in 4.
 Let us assume Claim \ref{wd} for now and find the various values of the differences $d(R_1,R_j)$ for different $j$.
 $$d(R_1,R_2)=2 ,\quad d(R_1,R_3)=3 ,\quad d(R_1,R_4)=0,\quad d(R_1,R_5)=4,\quad d(R_1,R_6)=4.$$
\end{example}

\begin{ex} \label{exdiff}
 Compute the other differences $d(R_i,R_j)$ for the knot coloring in Figure \ref{fig8diff}.
\end{ex}

\begin{figure}[!ht]
    \centering
    \includegraphics[width=7 cm]{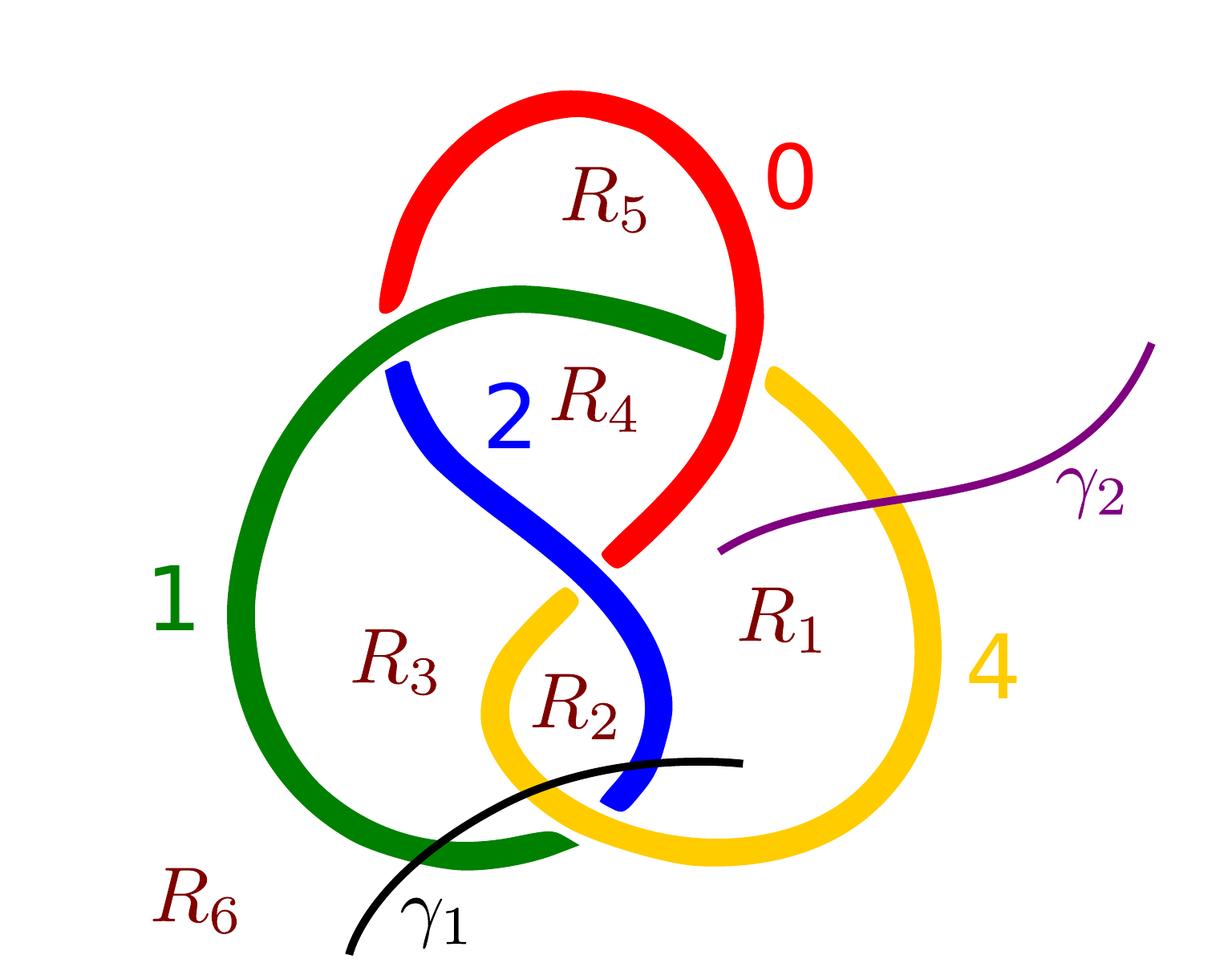}
    \caption{Differences in a coloring of the figure eight knot.}
    \label{fig8diff}
\end{figure}

\begin{proof}[ First proof of Claim \ref{wd}]

Note that in order to show this alternating sum is same for any two arcs with same endpoints, it is equivalent to show the alternating sum along any generic closed loop in the plane is zero
(concatenate one arc and the other's reverse). Let us begin by considering a generic simple closed loop $\gamma$ (i.e. no self intersections, it can intersect $D$).  We can think of $D$ together with the loop $\gamma$ as being a
diagram of a link of two components, one of them being the original knot $K$, and the other knot lying entirely below $K$. The latter is easily seen to be the unknot (as the diagram has no crossings), unlinked from $K$, by the way we chose crossings.
Hence we can isotope the unlinked unknot away from the original knot diagram, we can color it arbitrarily by some color. Since $n$-colorability is an invariant of the link, we know that the link diagram we started off with also has an $n$-coloring.
Suppose the color of the strand of this new unknot in some region $R_k$ is $c$, if we use the coloring equations check that the color we end up with once we traverse along $\gamma$ is the sum of $c$ and the alternating sum of the colors the
$\gamma$ crosses, and so this alternating sum has to be zero. Now if $\gamma$ was not simple, if $\gamma_1$ is a subloop of $\gamma$ which is simple then by our above discussion we see that the alternating sum along $\gamma_1$ is zero. Hence, we can
we can start deleting innermost loops from $\gamma$ without changing the alternating sum, and hence we get the alternating sum along any generic closed loop is zero. Note that we can vary the endpoints of the arcs in the same region by concatenating with an arc that does not cross $D$, and hence the alternating sum remains the same. Hence, it follows that $d(R_i,R_j)$ is well defined.
\end{proof}

 \begin{proof}[Second proof of Claim \ref{wd}]
An alternate approach to check well definedness would be to use the fact that any two such arcs are related by the following two moves:
going over a crossing, and encountering a birth/death (think about the non-generic instances and then the various ways we can perturb them to get to a generic arc), as illustrated in Figure \ref{moves}.
\begin{figure}[!ht]
    \centering
    \includegraphics[width=14 cm]{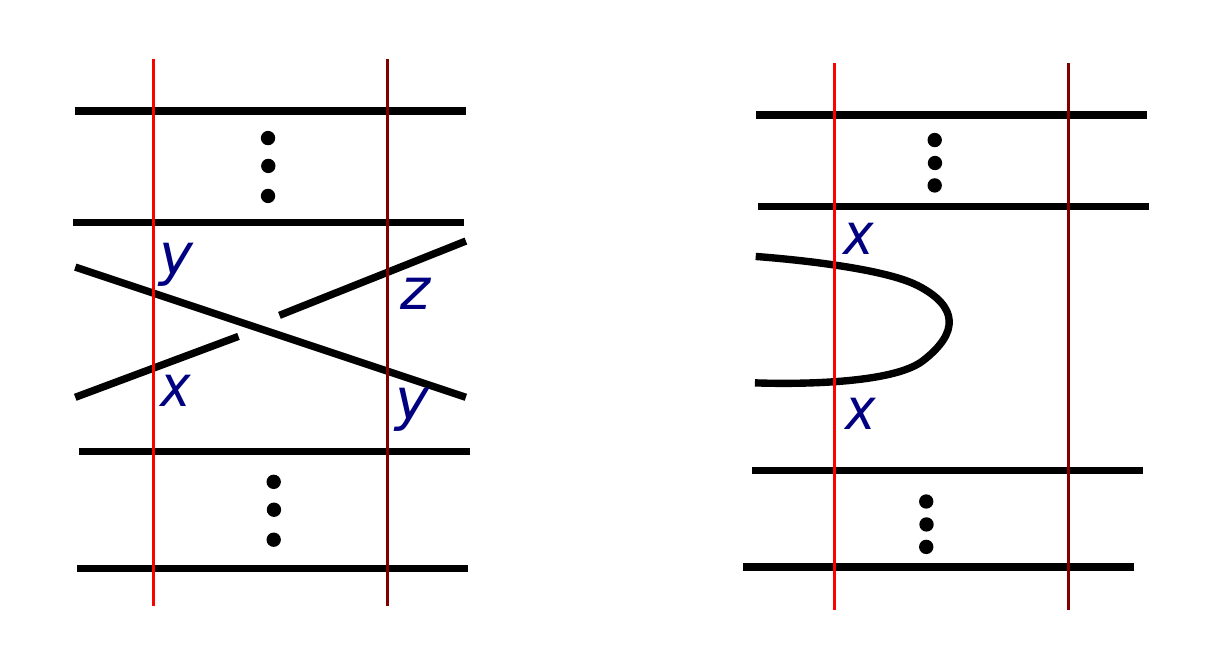}
    \caption{Moves relating a pair of isotopic arcs, when the knot diagram remains fixed.}
    \label{moves}
\end{figure}
We leave it to the reader to verify that the alternating sum remains the same under these moves.
\end{proof}

Let us now consider some checkerboard coloring of the knot diagram. For any triple $i,j,k$;  we can concatenate an arc from $R_i$ to $R_j$; and an arc from $R_j$ to $R_k$ to obtain an arc from $R_i$ to $R_j$. It follows that for any triple $i,j,k$, the differences satisfy 
 $$d(R_i,R_j)+d(R_j,R_k)=d(R_i,R_k) \text{ if } R_i \text{ and } R_j \text{ have same checkerboard coloring, and}$$ 
 $$d(R_i,R_j)-d(R_j,R_k)=d(R_i,R_k) \text{ if } R_i \text{ and } R_j \text{ have different checkerboard coloring.}$$ 
Conversely, if we are given such a collection $\{d(R_i,R_j)\}$ for all pairs of regions satisfying the above relations, 
then we obtain an $n$-coloring on the knot $K$: pick any two regions adjacent along a strand of the knot, and the difference between the two regions tells us the coloring on the strand. In other words, the data of an $n$-coloring is equivalent to the data of the differences.

 Since there are relations among the various $d(R_i.R_j)'s$, there is some redundancy;  we note that having $d(R_i.R_j)'s$ for adjacent regions is enough to recover the coloring. In fact, as we will see in the next section, to recover the coloring (up to translation by constant colorings), it is enough to only know the differences $d(R_i.R_j)'s$ among the crossing-adjacent shaded regions; and the nullspace of the (pre-)Goeritz matrix stores the information of the differences in a compact way.\\

\section{Bijections between solutions of coloring and Goeritz matrices}
In this section we demonstrate a bijection between solutions of the coloring matrix and Goeritz matrix of a diagram $D$. Suppose the shaded regions of $D$ are enumerated $R_1,...,R_s$.
Suppose we have an element $\vec v$ in the nullspace of the pre-Goertiz matrix $\tilde{G}$, i.e. $\tilde{G}\vec{v}=\vec{0}$. Let us define $d(R_i,R_j):=v_j-v_i$, and we see that $$d(R_i,R_j)+d(R_j,R_k)=v_j-v_i+v_k-v_j=v_k-v_i=d(R_i,R_k).$$

If these numbers $d(R_i,R_j)$ are differences coming from an actual coloring (this is true as we shall see) of the knot diagram, if we know a color on one of the strands, we can use these differences
to figure out the colors on all the other strands of the knot diagram. We will use this observation and assign a color to one strand, and given the numbers $d(R_i,R_j)$, we will assign colors
on all the other strands, and verify that this actually gives rise to a valid coloring. Just like in our discussion of the difference, we will need to make some choices and we need to check
that they are all give consistent coloring on a strand.

To this end, let us define an auxiliary knot diagram to be a modification of a knot diagram where we break up the overstrands at each crossing, and we
will draw a small rectangle at each crossing to keep track of which strand was the overcrossing, see Figure \ref{auxkd}.
\begin{figure}[!ht]
    \centering
    \includegraphics[width=9 cm]{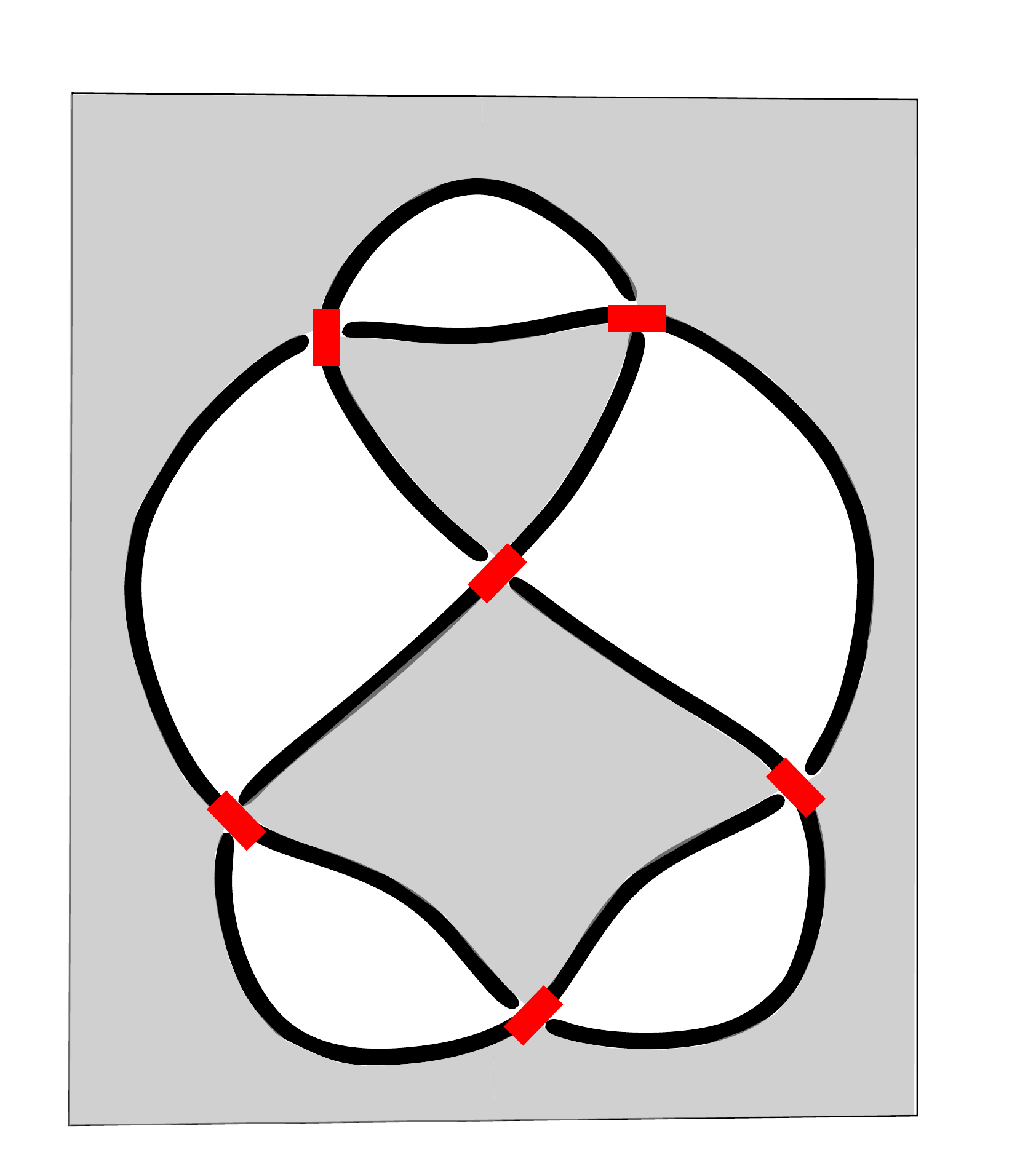}
    \caption{Auxiliary knot diagram for the knot diagram in Figure \ref{GM}.}
    \label{auxkd}
\end{figure}

We will define an \textbf{auxiliary $n$-coloring} on a knot diagram to be an assignment of colors to each of strands
of the auxiliary knot diagram, where at each crossing if the overstrands are labeled $w$ and $y$;
and the understrands are labelled $x$ and $z$, then $w+y=x+z$, i.e. the overstrands need not have the same labelling, see Figure \ref{auxkcolor}.
\begin{figure}[!ht]
    \centering
    \includegraphics[width=6 cm]{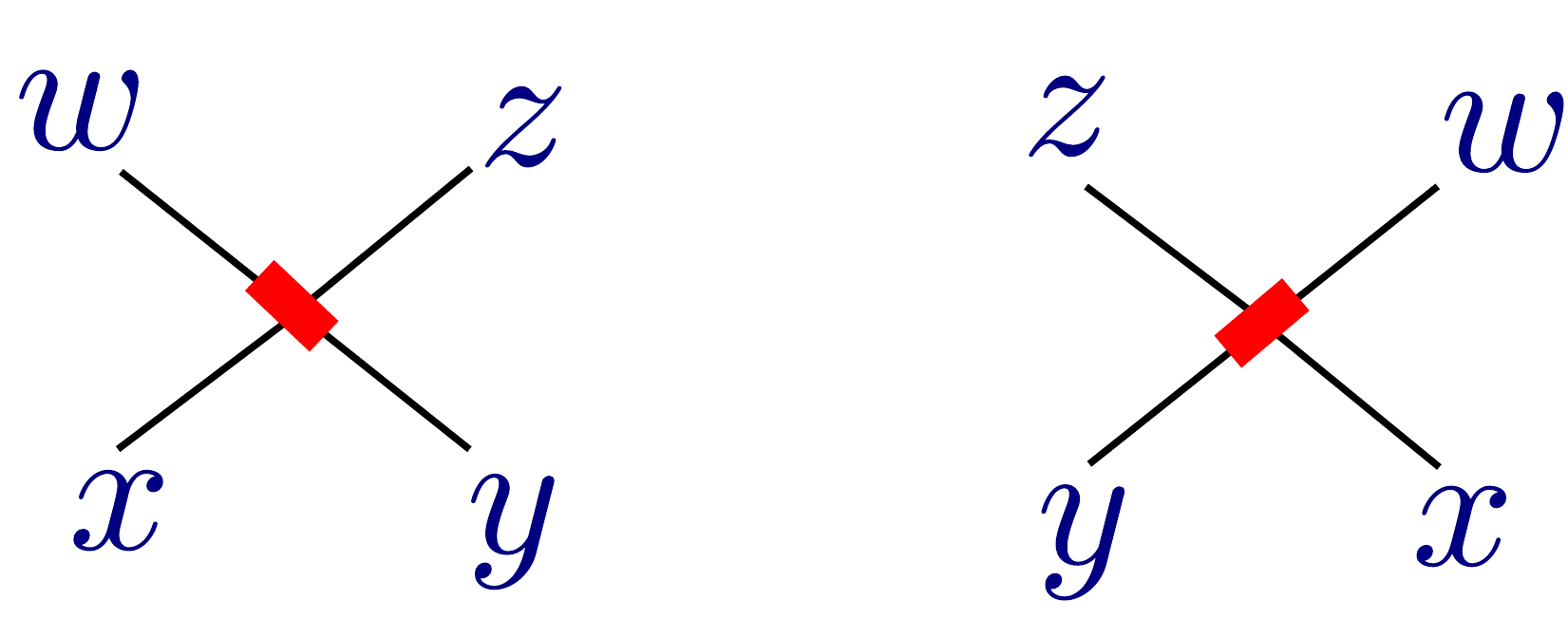}
    \caption{Auxiliary coloring equation: $w+y=x+z$.}
    \label{auxkcolor}
\end{figure}
If it turns out that at each crossing both the segments of 
the overstrands have the same labelling, then this auxiliary coloring gives us an actual coloring of the knot diagram. Note that we can define the difference for an auxiliary knot coloring exactly the same way we did for a knot coloring, and the second proof of Claim \ref{wd} carries over to show that this 
difference is well defined.

Let us pick a strand $\alpha_1$ in the auxiliary knot diagram and color it $x_1$. It will be part of the boundary of exactly one shaded region, let us call it $R_1$. As we go around
the boundary of $R_1$ in the counterclockwise direction, if we come across a crossing $c$ with the other shaded region being $R_j$; we will add $\eta(c)d(R_1,R_j)$ to the coloring of the subsequent strand, see Figure \ref{auxcolor}.
\begin{figure}[!ht]
    \centering
    \includegraphics[width=10 cm]{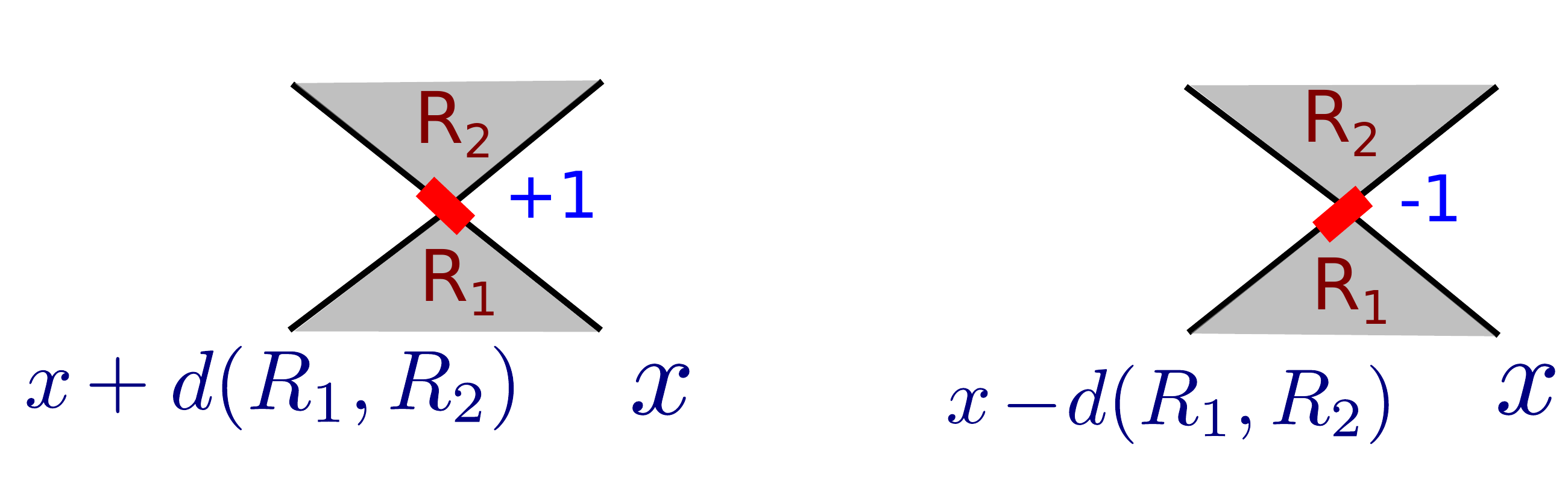}
    \caption{Assigning colors to each strand bounding a region in the auxiliary knot diagram.}
    \label{auxcolor}
\end{figure}
We need to make sure that when we go all the way across and come back to the strand $\alpha_1$ we still get the color $x_1$. Note that the color we get by going all the way across is
$$x_1+\sum_{c\text{ is a crossing of }R_1} \eta(c)d(R_1,R_{j(c)}),$$ where $R_{j(c)}$ denotes the other shaded region which has crossing $c$.
We observe that 
$$\sum_{c\text{ crossing of }R_1} \eta(c)d(R_1,R_{j(c)})=\sum_{c\text{ crossing of }R_1} \eta(c)(v_1-v_{j(c)}) =\sum_{j=2}^k\tilde{G}_{1,j}(v_j-v_1)$$
$$=(\sum_{j=2}^k-\tilde{G}_{1,j})v_1 + \sum_{j=2}^k\tilde{G}_{1,j}v_j=\tilde{G}_{1,1}v_1 + \sum_{j=2}^k\tilde{G}_{1,j}v_j=\sum_{j=1}^k\tilde{G}_{1,j}v_j=\tilde{G}\vec v=\vec 0$$
since $\vec v$ is in the null space of the pre-Goeritz matrix $\tilde{G}$.

We can do the exact same thing for each of the regions by picking one strand and arbitrarily assigning a color, and coloring all the other strands by adding $\eta(c)d(R_i,R_j)$ to the coloring
as we pass across a crossing $c$ with shaded regions $R_i$ and $R_j$, and we obtain an auxiliary coloring of the knot diagram.

Let us now choose a collection of $s-1$ crossings joining the $s$ shaded regions $R_1,...,R_s$, i.e. given any $i$ and $j$, we can draw an arc from a point in the interior of region $R_i$ to the
a point in the interior of region $R_j$ so that the arc passes between regions only in arbitrarily small neighborhoods of the chosen crossings. Starting with a color $x_1$ on a strand $\alpha_1$
in the region $R_1$, we obtain an auxiliary $n$-coloring on the entire knot diagram by choosing to color the overstrands at the chosen crossings by the same element of $\mathbb{Z}/n\mathbb{Z}$, 
and extending this to each of the strands partially bounding a region by the procedure explained in the preceding two paragraphs. By construction, the colorings on each of the overstrands agree
on the chosen $s-1$ crossings, and we show below they agree on every other crossing as well, thereby giving us an actual knot coloring. 

Let us denote by $\delta(R_p,R_q)$ the difference between the regions $R_p$ and $R_q$ of this auxiliary $n$-coloring.
Suppose $c$ is any crossing in the knot diagram, with the shaded regions $R_i$ and $R_j$ incident on it. We have an arc $\gamma$ from a point in the interior of region $R_i$ to the
a point in the interior of region $R_j$,  traversing the regions $R_{l_1}=R_i,R_{l_2},...,R_{l_r}=R_j$, so that it goes between regions $R_{l_k}$ and $R_{l_{k+1}}$ in an arbitrarily small
neighborhoods of the chosen crossings. Then we have $$\delta(R_i,R_j)=\sum_{k=1}^{r-1} \delta(R_{l_k},R_{l_{k+1}}) \quad\text{  (by choosing the arc $\gamma$ to compute the difference)} $$
$$ =\sum_{k=1}^{r-1} d(R_{l_k},R_{l_{k+1}}) \quad\text{    (by construction of the auxiliary coloring)} $$
$$=d(R_i,R_j) \quad\quad\text{ (telescoping sum since $d(R_p,R_q)=v_q-v_p$).}$$
It follows that at the crossing $c$ the colorings on the overstrands agree, and since $c$ is arbitrary, this is true at any crossing. Thus, we have obtained a valid knot coloring with its 
difference for shaded regions $\delta(R_i,R_j)$ agreeing with $d(R_i,R_j)$ we defined earlier. We should note that this implies we would obtain the same knot coloring independent of the choice of
which of the $s-1$ crossings we chose earlier. Also, if instead of coloring the strand $\alpha_1$ with $x_1$, we colored $\alpha_1$ with the color $x_1+a$, then the above procedure would give a new coloring where every strand would get the color $a$ plus
the original coloring.

Conversely, we note that if we start with a valid coloring, we can look at the differences $d(R_i,R_j)$ among the shaded regions, and they would satisfy the Goeritz relations $$\sum_{c\text{ crossing
of }R_i} \eta(c)d(R_i,R_{j(c)})=0$$ for each region $R_i$ , which in turn would correspond to a solution of the pre-Goeritz matrix, once we choose a value for a region, say $R_1$.

 Hence, under this procedure semitrivial (respectively non-semitrivial) solutions of the pre-Goeritz matrix correspond to the semitrivial (respectively non-semitrivial) solutions of the pre-coloring matrix. 
 More concretely, we have:

   \begin{thm}\label{thmA}

 Suppose we have a knot diagram $D$ for a knot $K$ with $c$ crossings and $s$ shaded regions, let us choose a strand $\alpha_1$ of the knot diagram partially bounding a shaded region $R_1$. Let's suppose we delete the first column and first row of 
 the pre-coloring matrix $\tilde{C}$ (respectively pre-Goeritz matrix $\tilde{G}$)
 to obtain coloring matrix $C$ (respectively Goeritz matrix $G$).
 Then for any $n\in \mathbb{N}$ there is a bijection among the collection of: 
 \begin{enumerate}[label=(\roman*)]
     \item Non-trivial solutions $\tilde{v}\in (\mathbb{Z}/n\mathbb{Z})^c$ of $\tilde{C}$ with first entry 0.
     \item Non-trivial solutions $v\in (\mathbb{Z}/n\mathbb{Z})^{c-1}$ of $C$.
     \item Non-trivial solutions $w\in (\mathbb{Z}/n\mathbb{Z})^{s-1}$ of $G$.
     \item Non-trivial solutions $\tilde{w}\in (\mathbb{Z}/n\mathbb{Z})^s$ of $\tilde{G}$ with first entry 0.
 \end{enumerate}
     Moreover the above bijection between Nul($C$) and Nul($G$) is linear, and so we get an isomorphism of these null spaces of $\mathbb{Z}/n\mathbb{Z}$-modules (for $n$ prime, this means a vector space isomorphism). In particular, for prime $n$, the mod $n$ nullity (i.e. dimension of the null space) of both these matrices are equal.
  \end{thm}

  \begin{ex}\label{exA} Complete the proof of the above theorem by checking linearity of the bijection between the solution spaces.
\end{ex}
  
For a prime $n$, the mod $n$ nullity of either the coloring or Goeritz matrix completely determines how many $n$-colorings there are of a knot $K$, and is an invariant of the knot, which we will refer to as the $n$-\textbf{nullity} 
of $K$. We remark that if $n$ is composite, then knowing the cardinality of Nul($C$) is not enough to understand the structure of Nul($C$) as a $\mathbb{Z}/n\mathbb{Z}$-module. But when $n$ is prime, Nul($C$) is a vector space, and we completely understand it once we know the dimension.

We note a few consequences of Theorem \ref{thmA} (combined with earlier propositions):
\begin{enumerate}
    \item  $n$ divides $\det C \Longleftrightarrow$ ${C}$ has a non-trivial solution in $\mathbb{Z}/n\mathbb{Z}\Longleftrightarrow$ $K$ has a non-trivial $n$-coloring ${G}\Longleftrightarrow$ ${G}$ has a non-trivial solution in $\mathbb{Z}/n\mathbb{Z}\Longleftrightarrow$ $n$ divides $\det G$.
    \item The Goeritz determinant is an invariant of the knot, and moreover equals the coloring determinant.
     \item For any $n$, the null-space of a Goeritz matrix as a $\mathbb{Z}/n\mathbb{Z}$-module is a knot invariant, which is isomorphic to the null space of any coloring matrix of $K$.

\end{enumerate}
Thus, we can find the  determinant of knot; the space of $n$-colorings determinant from a Goeritz matrix, which is frequently smaller than a coloring matrix.

\section{Colorings of Pretzel knots}
In this final section, we will use ideas from the last few sections and find determinant and $n$-nullity of pretzel knots.
They have been previously computed pretzel knots with up to four twist regions \cite{P1,P2} using coloring matrices.\\

Recall that a pretzel knot $P(q_1, q_2,...,q_m)$ has $m$ twist regions joined up as illustrated in Figure \ref{pret}, where there are $q_i$ (which can be both positive and negative) half-twists in the $i$-th region. See Figure \ref{p33m3} for the diagram of the pretzel knot $P(3,3,-3)$.

\begin{figure}[!ht]
    \centering
    \includegraphics[width=8 cm]{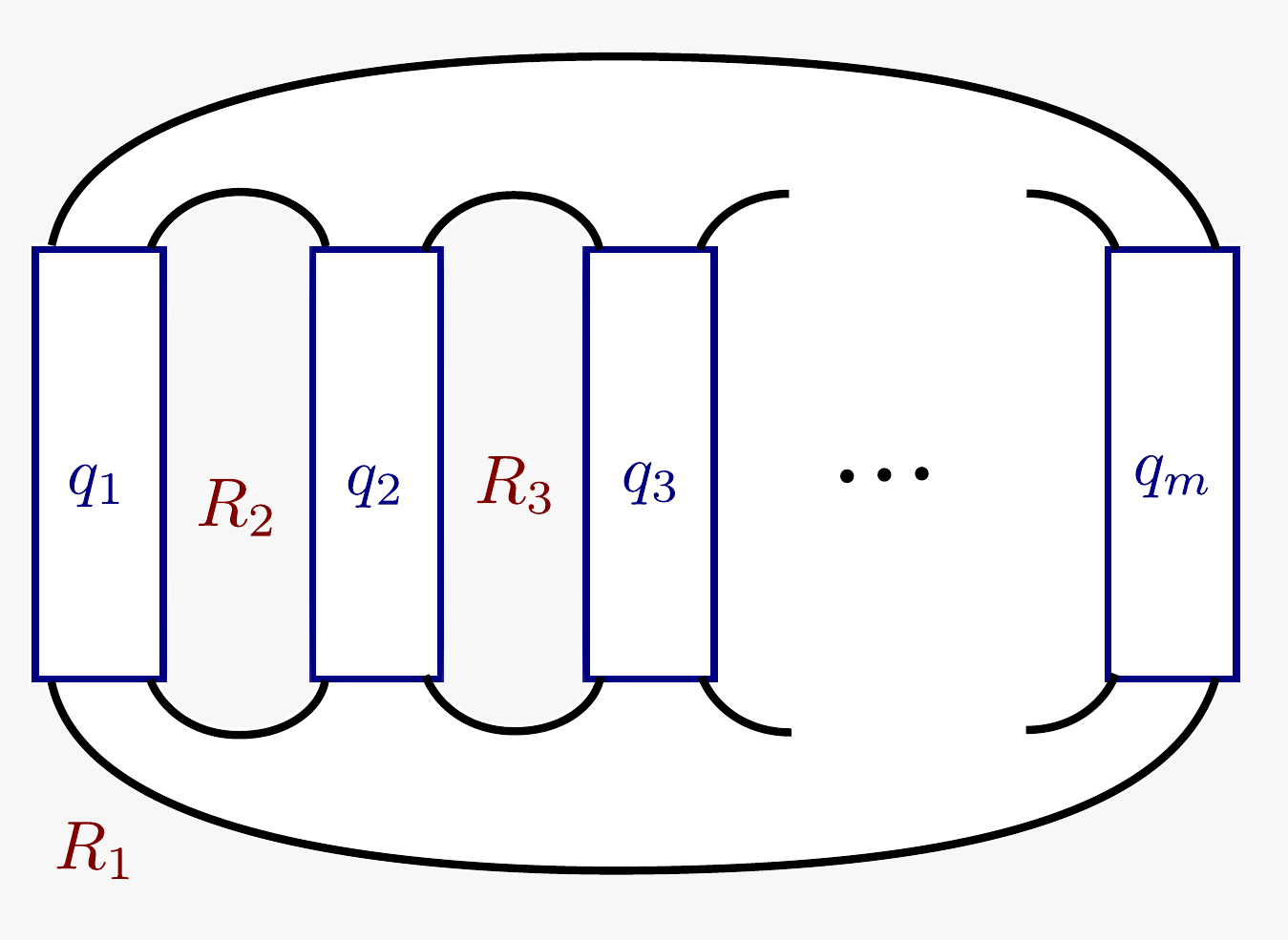}
    \caption{Pretzel Knot $P(q_1,...,q_m)$.}
    \label{pret}
\end{figure}

Note that for the pretzel knot $P(q_1, q_2,...,q_m)$ the coloring matrix from the above diagram will be of the order $Q\times Q$, where $Q=|q_1|+...+|q_m|$, where as a Goeritz matrix will be a $(m-1)\times (m-1)$ matrix.
It turns out that it seems it is easier to compute the determinant and nullity of pretzel knots from another linear system, constructed with the differences, as indicated below. An interested reader may work out determinant and $n$-nullity from a Goeritz matrix (and even from a coloring matrix, if feeling particularly adventurous).

Let $d_i$ denote the difference $d(R_i,R_{i+1})$ between shaded regions $R_i$ and $R_{i+1}$. Note that for each twist region if there are $q$ half-twists, and the difference is $d$, then we can figure out the colors on all the strands if we know a color on the leftmost strand.  We observe that the difference between the colors of the top left (respectively right) and the bottom left (respectively right) strand in the $i$-th twist region is $q_id_i$. Also, note that the top (respectively bottom) right strand of the $i$-th twist region is exactly the same as the top (respectively bottom) leftt strand of the $(i+1)$-th twist region. Thus, for adjacent twist regions we must have the increase in the vertical direction must be the same, so we have $q_id_i=q_{i+1}d_{i+1}$ for all $i$. Moreover if we drew an arc horizontally traversing just below (see figure) each of the twist region regions we see that $-(d_1+d_2+...+d_m)=0$ , since the leftmost and rightmost strands are the same and must have the same color. We illustrate the above discussion with the explicit example of of the pretzel knot $P(3,3,-3)$ in Figure \ref{p33m3}. 

\begin{figure}[!ht]
    \centering
    \includegraphics[width=9 cm]{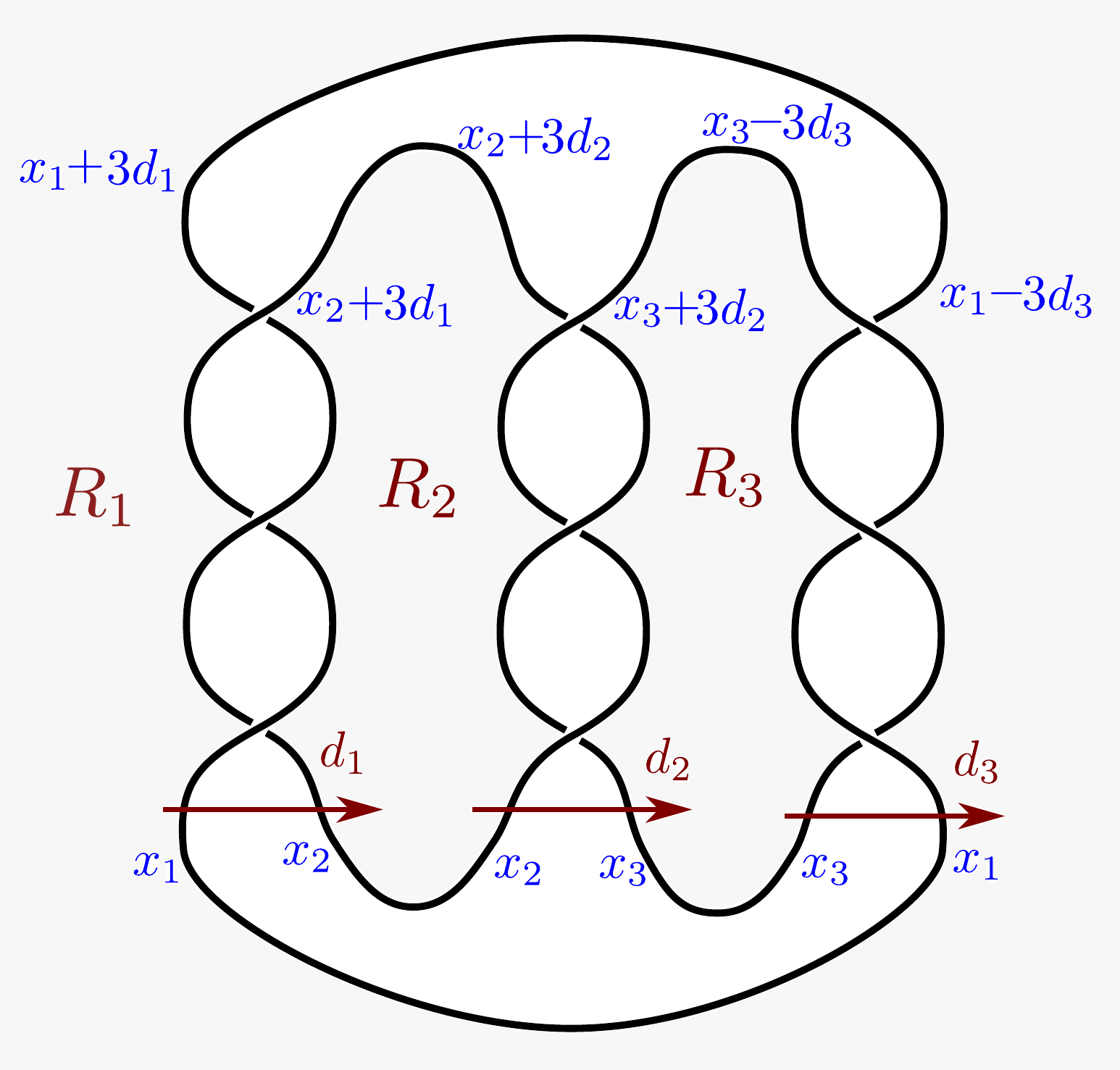}
    \caption{Pretzel Knot $P(3,3,-3)$.}
    \label{p33m3}
\end{figure}

Thus we obtain the following collection of linear equations in $\mathbb{Z}/n\mathbb{Z}$:
$$d_1+d_2+...+d_m=0$$
$$-q_1d_1+q_2d_2=0$$
$$...$$
$$-q_1d_1+q_md_m=0$$

Let us write out the above system in matrix form (with $d_i$'s being the variables) $A\vec d=\vec 0$, where the matrix $A$ is:

$$\spalignmat{1 1 1 1 . . . 1 ; -q_1 q_2 0 0 . . . 0; -q_1 0 q_3 0  . . . 0; 
. . . . . , , ,  .; . . . . , , .  , , .;. . . . , , , . , .;
-q_1 0 0 0 . . . q_m}$$

\begin{ex}\label{expr} For a pretzel knot diagram of $P(q_1, q_2,...,q_m)$, pick a strand and name it $\alpha_1$. Given a solution $\vec d$ to the linear system described above, and any labelling of the strand $\alpha_1$ by some $x_1 \in (\mathbb{Z}/n\mathbb{Z})$, show that there is a unique extension to coloring on the entire knot diagram by using the solution $\vec d$. Moreover show that this defines a linear bijective correspondence between the nullspace of the matrix $A$ reduced modulo $n$, and the space of $n$-colorings of the knot diagram (i.e. nullspace of any coloring matrix).
\end{ex}

It follows that we can use the matrix $A$ to compute the determinant and $n$-nullity of $P(q_1, q_2,...,q_m)$, and this method seems to be the easiest way of computing them.

\begin{claim}\label{det}
The determinant of the $m\times m$ matrix $A$ is given by $$q_1q_2...q_m (\frac{1}{q_1}+\frac{1}{q_2}+...+\frac{1}{q_m})$$
\end{claim}

\begin{proof}

We will prove the claim by induction on $m$.
\begin{ex}
Check the base cases for $m=2,3.$
\end{ex}

In order to compute the determinant of $A$, we use Laplace expansion along the last row.
$$ \det A=(-1)^{m+1} q_1 \det \spalignmat{ 1 1 1 . . . 1 ;  q_2 0 0 . . . 0;  0 q_3 0  . . . 0; 
 . . . . , , ,  .;  . . . , , .  , , .; . . . , , , . , .;
 0 0 0 . . q_{m-1} 0}+  q_m \det \spalignmat{1 1 1  . . . 1 ; -q_1 q_2 0  . . . 0; -q_1 0 q_3   . . . 0; 
. . . . . , ,  .; . . . . , , .   , .;. . . . , , . .  ;
-q_1 0 0 .  . . q_{m-1}}$$

We observe that by expanding along the last column, we can compute the determinant of the first matrix easily since we get a diagonal matrix:
$$\det \spalignmat{ 1 1 1 . . . 1 ;  q_2 0 0 . . . 0;  0 q_3 0  . . . 0; 
 . . . . , , ,  .;  . . . , , .  , , .; . . . , , , . , .;
 0 0 0 . . q_{m-1} 0}= (-1)^{m-1}\det \spalignmat{ q_2 0  . . . 0;  0 q_3   . . . 0; 
 . .  . , , ,  .;  . .  , , .  , , .; . .  , , , . , .;
 0 0  . . . q_{m-1}}= (-1)^{m-1} q_2...q_{m-1}  $$

Inductively, we know what the second determinant in the above expansion is. 
$$\det \spalignmat{1 1 1  . . . 1 ; -q_1 q_2 0  . . . 0; -q_1 0 q_3   . . . 0; 
. . . . . , ,  .; . . . . , , .   , .;. . . . , , , .  ;
-q_1 0 0 .  . . q_{m-1}}= q_1q_2..q_{m-1} (\frac{1}{q_1}+\frac{1}{q_2}+...+\frac{1}{q_{m-1}})$$

Combining them we obtain
$$\det A= (-1)^{m+1} q_1 (-1)^{m-1}q_2...q_{m-1} + q_mq_1q_2...q_{m-1} (\frac{1}{q_1}+\frac{1}{q_2}+...+\frac{1}{q_{m-1}})$$
$$=q_1q_2...q_{m-1}+ q_1q_2...q_m(\frac{1}{q_1}+\frac{1}{q_2}+...+\frac{1}{q_{m-1}})=q_1q_2...q_m (\frac{1}{q_1}+\frac{1}{q_2}+...+\frac{1}{q_m})$$

\end{proof}

Let us now assume $n$ is a prime and compute mod $n$ nullity of $A$. 
We make a few observations:
\begin{claim} \label{colpr}
If all the $q_i$'s are coprime to $n$ then the mod $n$ nullity of $A$ is either 1 or 0, depending on whether n divides $\det A= q_1q_2...q_m(\frac{1}{q_1}+\frac{1}{q_2}+...+\frac{1}{q_m})$ or not.

\end{claim}

\begin{proof}
Note that the submatrix $A_{1,1}$ obtained by deleting the first row and first column is a full rank diagonal matrix and so has mod $n$ rank $m-1$, since we are assuming $n$ does not divide any of $q_2,...,q_n$. It follows that the mod $n$ rank of $A$ is either $m-1$ or $m$ (which by the rank nullity theorem is equivalent to saying the mod $n$ nullity is 1 or 0). The proof is completed by using the Invertible Matrix Theorem, that mod $n$ nullity of $A$ is 0 iff $\det A=0$ modulo $n$. 
\end{proof}

\begin{claim}\label{nullpr}
If some of the $q_i$'s are divisible by $n$, then the mod $n$ nullity of $A$ is the total number of $q_i$'s divisible by $n$, minus $1$.
\end{claim}

If some of the $q_i $'s are divisible by $n$, we may assume $n$ divides $q_1$ (note that there is cyclic symmetry for pretzel knots, $P(q_1,q_2,...,q_{m-1},q_m)$ is isotopic to $P(q_2,q_3,...,q_{m},q_1)$). In this case, the claim is equivalent to:\\
If $n$ divides $q_1$,then the mod $n$ nullity of $A$ is the number of $q_i$'s (apart from $q_1$) divisible by $n$.
\begin{proof}
When $A$ is reduced modulo $n$, it becomes a upper triangular matrix, and the rank is the number of pivots (i.e. the number of $q_i$'s not divisible by $n$, plus 1 ), and the nullity is the number of 0's in the main diagonal (the number of "other" $q_i$'s divisible by $n$).
\end{proof}

\begin{ex}\label{exnull}
Give a direct proof Claim \ref{nullpr} (without using that $n$-nullity is invariant of the knot), i.e. show the result is true when $q_1$ is coprime to $n$. (Hint: Use row operations).
\end{ex}

To summarize, if none of the $q_i$'s are divisible by $n$, the pretzel knot can have at most one coloring and the determinant determines whether there is one. In cases that some of the $q_i$'s are divisible by $n$, for any $n$-coloring
\begin{enumerate}
    \item the colors of all the strands in the $i$-th twist region will be the same when the corresponding $q_i$ is coprime to $n$;
    \item for those $q_i$ which are divisible by $n$, the colors of the strands in the $i$-th twist region can be different, and these sort of correspond to the free variables, except one has to remember that the colors of the very last such twist region is determined (by the coloring on the first twist region), and this constraint is where we get the "minus 1" in the formula for $n$-nullity.
\end{enumerate}

\end{document}